\documentclass[11pt]{amsart}
\usepackage{amsmath,amssymb}
\usepackage{hyperref}
\usepackage{enumerate}
\usepackage{bm}

  \usepackage{graphicx}
  \usepackage{latexsym}
  \usepackage[all]{xy}
  \usepackage{amsfonts}
  \usepackage{xcolor}

\newtheorem{proposition}{Proposition}[section]
\newtheorem{lemma}[proposition]{Lemma}
\newtheorem{corollary}[proposition]{Corollary}
\newtheorem{theorem}[proposition]{Theorem}

\theoremstyle{definition}
\newtheorem{definition}[proposition]{Definition}

\theoremstyle{remark}
\newtheorem{remark}[proposition]{Remark}
\newtheorem{remarks}[proposition]{Remarks}

\newcommand{\thlabel}[1]{\label{th:#1}}
\newcommand{\thref}[1]{Theorem~\ref{th:#1}}

\newcommand{\selabel}[1]{\label{se:#1}}
\newcommand{\seref}[1]{Section~\ref{se:#1}}

\newcommand{\lelabel}[1]{\label{le:#1}}
\newcommand{\leref}[1]{Lemma~\ref{le:#1}}
\newcommand{\prlabel}[1]{\label{pr:#1}}

\newcommand{\colabel}[1]{\label{co:#1}}
\newcommand{\coref}[1]{Corollary~\ref{co:#1}}
\newcommand{\relabel}[1]{\label{re:#1}}
\newcommand{\reref}[1]{Remark~\ref{re:#1}}

\newcommand{\delabel}[1]{\label{de:#1}}
\newcommand{\deref}[1]{Definition~\ref{de:#1}}

\newcommand{\Vect}{{\sf Vec}}

\newcommand{\id}{{\sf id}\,}
\newcommand{\coev}{{\rm coev}}
\newcommand{\ev}{{\rm ev}}

\newcommand{\Hom}{{\rm Hom}}

\newcommand{\Rep}{{\sf Rep\hbox{-}}}

\def\ot{\otimes}

\def\id{\textrm{{\small 1}\normalsize\!\!1}}

\def\CC{{\mathbb C}}

\def\NN{{\mathbb N}}

\def\ZZ{{\mathbb Z}}

\newcommand{\Aa}{\mathcal{A}}
\newcommand{\Bb}{\mathcal{B}}
\newcommand{\Cc}{\mathcal{C}}
\newcommand{\Dd}{\mathcal{D}}
\newcommand{\Ee}{\mathcal{E}}

\newcommand{\Hh}{\mathcal{H}}

\newcommand{\Mm}{\mathcal{M}}

\newcommand{\Pp}{\mathcal{P}}
\newcommand{\Qq}{\mathcal{Q}}

\def\text#1{{\rm {\rm #1}}}

\begin{document}
\title[The Witt group of a braided monoidal category]{The Witt group of a braided monoidal category}

\author{Isar Goyvaerts}
\address{Department of Mathematics, Faculty of Engineering, Vrije Universiteit Brussel, Pleinlaan 2, B-1050 Brussel, Belgium}
\email{igoyvaer@vub.ac.be}
\author{Ehud Meir}
\address{Centre for Symmetry and Deformation, Institut for Matematiske Fag, Universitetsparken 5, 2100 K\o benhavn \O, Denmark}
\email{meir@math.ku.dk}
\begin{abstract}
We develop the Witt group for certain braided monoidal categories with duality. In case of a braided fusion category over an algebraically closed field of characteristic zero, 
we explictly describe this structure. We then use this description to prove that this tool provides an invariant for finite isocategorical groups. 
As an application, we show that all groups of order less than 64 are categorically rigid.
\end{abstract}

\maketitle
\section*{Introduction}\selabel{Intro}
In this paper, we develop the Witt group for $k$-linear braided monoidal categories with duality, where $k$ is a commutative ring with $2$ being an invertible element (categories of finite-dimensional complex representations of finite groups will be our main concern). Our device finds its roots in the theory of quadratic forms.
In 1937, Ernst Witt introduced a group structure -and even a ring structure- on the set of isometry classes of anisotropic quadratic forms over an arbitrary field $k$.
This object is now called the Witt group of $k$ and often denoted by $W(k)$.
Since then, Witt's construction for fields has been generalized in many directions; to commutative rings  (cf. \cite{Bass})
and to various types of categories with duality (cf.  \cite{Quebbemann}), to name only but a few.
Rather than focusing on this rich tradition, let us turn towards the setting in which the content of the present paper is situated.
\\Let $k$ be a commutative ring with $\frac{1}{2}\in k$ and let $\Cc$ be a $k$-linear, braided monoidal category with duality.
We propose to study ``quadratic objects'' in $\Cc$ that are ``non-singular'' (they coincide with the classical notion of non-degenerate quadratic module,
when $\Cc$ is taken to be the category of finitely generated projective $k$-modules). Following the treatments in \cite{Bass} and \cite{Quebbemann},
we construct a group whose elements are isometry classes of such objects, modulo the image of the so-called hyperbolic functor.
This group -which will be called the \textit{Witt group of $\Cc$} and denoted $W(\Cc)$, by analogy with the classical theory-
can be endowed with the structure of a commutative, unital ring, provided the braiding is a symmetry. This is done in \seref{QuadraForm}.
\\In case $\Cc$ is a braided fusion category over an algebraically closed field of characteristic $0$, we give a description of $W(\Cc)$ in \seref{semisimple}. 
\\Next, in \seref{examples}, we compute concrete examples arising from finite groups. In particular, we provide a description of $W({\Rep}G)$, resp. $W(\Rep{D(G)})$, the Witt ring of the category of finite-dimensional complex representations of a finite group $G$, resp. of the Drinfeld double of $G$.
%
\\We then use these descriptions to address the question of non-isocategoricity of finite groups.
\\In \cite{EtiGel}, Etingof and Gelaki define two finite groups $G$ and $H$ to be isocategorical if ${\Rep}G$ and ${\Rep}H$ are equivalent as monoidal categories
without regard of the symmetric structure (if they are equivalent as symmetric monoidal categories, $G$ and $H$ must be isomorphic, cf. \cite{Deligne} for instance).
Now, the property of two groups being isocategorical is much stronger than the property of their categories of representations having isomorphic Grothendieck rings.
For example, it is known that the two non-abelian groups of order 8 are not isocategorical,
although the corresponding Grothendieck rings (i.e. the character rings) are isomorphic (cf. \cite{TambaraYam} e.g.).
This raised the question whether isocategorical groups must be isomorphic. This is not the case; cf. \cite{EtiGel}, \cite{DavydovGal} and \cite{IzuKosa},
where examples of isocategorical, yet non-isomorphic groups are constructed. 
In spite of this fact, all finite groups isocategorical to a given group can be classified in group-theoretical terms (cf. \cite[Theorem 1.3]{EtiGel}).
However, it is not always easy to check whether two arbitrary finite groups are isocategorical or not, even when using this classification result.
In \seref{IsoCat}, we show that $W(\Rep{?})$ (and $W(\Rep{D(?)})$) are invariants for finite isocategorical groups. 
We remark that this fact can as well be proven using a theorem of Mason and Ng (namely Theorem 4.1. in their paper \cite{MasonNg}). 
We provide here a different proof, based on the group-theoretical description of finite isocategorical groups by Etingof and Gelaki from \cite{EtiGel}.
\\As an application, we provide examples of finite groups that have isomorphic character rings and isomorphic Witt rings,
yet are not isocategorical, thus showing that the character ring and the Witt ring of $G$ are not enough to distinguish the isocategoricity type of $G$.
Shimizu proved in \cite{Shimizu} that if two groups are isocategorical, then their number of elements of order $n$ is equal for every natural $n$.
We will present in \seref{Order 32} an example of two groups having isomorphic Grothendieck rings, Witt rings, and the same number of elements of any given order,
yet not being isocategorical, thus proving that these invariants are still not enough to distinguish isocategoricity classes of finite groups. Combining the results of Etingof and Gelaki, the Grothendieck and Witt rings, 
one obtains that all groups of order less than 32 are categorically rigid 
(a finite group $G$ is called categorically rigid if any group isocategorical to $G$ is actually isomorphic to $G$); a result which has been proved -by different means- 
by Shimizu in \cite{Shimizu}. In \seref{Order 32}, examining the above mentioned pair of groups, we show that all groups of order 32 are categorically rigid as well. Finally, again combining the result of Etingof and Gelaki, the Grothendieck and Witt rings, one checks that all groups of order greater than 32 and less than 64 are categorically rigid as well, hereby concluding that the smallest known example of two isocategorical, yet non-isomorphic groups appearing in literature
(which is of order 64, as described by Izumi and Kosaki in \cite{IzuKosa}), indeed is the smallest possible one.

\section{Notational conventions}
For any two objects $X,Y$ in a category $\Aa$, we denote the Hom-set from $X$ to $Y$ as $\Hom_\Aa(X,Y)$ or shortly by $\Hom(X,Y)$ 
if there can be no confusion about the category $\Aa$. The identity morphism on $X$ is denoted by $1_X$.
Let $\Aa$ and $\Bb$ be two categories. A functor $F:\Aa\to \Bb$ will be the name for a covariant functor; 
it will only be a contravariant one if it is explicitly mentioned. 
The symbol $\id_{\Aa}$ denotes the identity functor on $\Aa$.
\\$\Cc=(\Cc,\ot, I)$, will be our notation for a (strict) monoidal category, where $I$ is the unit object of $\Cc$.

\section{Quadratic Forms in Braided Monoidal Categories}\selabel{QuadraForm}
Throughout the paper, $k$ will be a commutative ring such that $\frac{1}{2}\in k$.
In the classical theory, a quadratic form $q$ on a finitely generated projective module $P$ over $k$
is defined as a map $q:P\to k$ satifying
\begin{itemize}
\item $q(ax)=a^{2}q(x),~~~~\forall x\in P, a\in k$
\item $b_{q}(x,y):=q(x+y)-q(x)-q(y)$ is a (necessarily symmetric) bilinear form on $P$.
\end{itemize}
If $b$ is a bilinear form on $P$, then $q_{b}:P\to k$ defined by $q_{b}(x):=b(x,x)$ is a quadratic form.
\\Denoting by $b^{t}$ the bilinear form defined by $b^{t}(x,y)=b(y,x)$, one has the following
\begin{lemma}
If $b$ is a bilinear form and $q$ is a quadratic form on $P$, then
\begin{center}
$b_{q_{b}}=b+b^{t}$ and $q_{b_{q}}=2q$.
\end{center}
\end{lemma}
Using the fact that $P$ is finitely generated projective, one has the following observation:
\begin{itemize}
\item Every quadratic form is of the form $q_{b}$ for a suitable $b$.
\item $q_{b}=q_{b'}$ precisely when $b-b'$ is alternating.
\end{itemize}
Classically, the construction of the Witt group of $k$ is roughly outlined as follows
(we refer to \cite{Bass} for the terminology we use here):
\begin{enumerate}
\item[i)] A suitable definition of \textit{non-singular} quadratic module is given. 
It coincides with the notion of non-singular quadratic spaces in the classical theory when $k$ is a field.
\item[ii)] One defines the \textit{orthogonal sum} of two quadratic modules and proves that the category of non-singular quadratic modules
(with \textit{isometries} as morphisms) is a category \textit{with product}, induced by this orthogonal sum. Let us denote this category with product as $\Qq$.
We remark that we do not mean here a product in the categorical sense. Instead, we think of the product as a functor $\oplus:\Qq\times\Qq\rightarrow \Qq$.
\item[iii)] The \textit{hyperbolic} functor $\Hh$ is defined between the two following categories with product: $\Pp$; the category of finitely generated, projective $k$-modules
(considered with isomorphisms in $\Mm_{k}$ as morphisms and the direct sum $\bigoplus$ of $k$-modules) and $\Qq$.
It is proved that $\Hh$ preserves the product and that it is \textit{cofinal}, i.e. for every $M\in\Qq$ there is an $N\in\Qq$ and $L\in\Pp$ such that $M\oplus N\cong \Hh(L)$. 
The Witt group of $k$ is then defined as ${\rm Coker}(\mathbb{K}_{0}\Hh)$.
\end{enumerate}
We would like to define the Witt group for more general braided monoidal categories then the category ${\rm Proj}_k$ of finitely generated projective $k$-modules.
For this, we consider a rigid braided monoidal $k$-linear category $\Cc$, in which the isomorphism classes form a set.
Some definitions are in order about this terminology (to lighten notation, the category $\Cc$ is considered to be strict; this without being harmful to generality, 
by Mac Lane's coherence theorem):
The ring $k$ is a fixed commutative ring satisfying the assumption $\frac{1}{2}\in k$.
The fact that $\Cc$ is $k$-linear means that the hom sets in $\Cc$ are $k$-modules and that the composition is $k$-bilinear.
Rigidity of $\Cc$ means that $\Cc$ has duality: 
for every object $M\in\Cc$ there is an object $M^*\in \Cc$ and two morphisms ${\rm ev}_M:M^*\ot M\rightarrow I$ and ${\rm coev}_M:I\rightarrow M\ot M^*$ 
(where $I$ is the unit object in $\Cc$) 
such that the two compositions
$$M\longrightarrow M\ot M^*\ot M\rightarrow M\textrm{, }$$ 
$$M^*\rightarrow M^*\ot M\ot M^*\rightarrow M^*$$ 
are the identity morphisms. We call the tuple $(M,M^*,{\rm ev}_M,{\rm coev}_M)$ a dual pair in $\Cc$.
If every object in $\Cc$ has a dual, then the assignment $M\mapsto M^*$ can be extended to a contravariant functor from $\Cc$ to $\Cc$ (see Proposition XIV.2.2. in \cite{K}).
Strictly speaking, $M^*$ is called a right dual to $M$. The left dual $^*M$ is defined in a similar way, where we invert the order of the tensor product.
Since we consider only braided categories, we will see in Lemma \ref{gam} that left and right duals are naturally isomorphic.

The duality axioms imply that the functors of tensor multiplication by $M$ and by $M^*$ are adjoint to each other:
\begin{equation}\label{dual1}\Hom_{\Cc}(M^*\ot X,Y)\cong \Hom_{\Cc}(X,M\ot Y)\textrm { and}\end{equation}
\begin{equation}\label{dual2}\Hom_{\Cc}(X\ot M,Y)\cong \Hom_{\Cc}(X,Y\ot M^*)\end{equation}
We thus have for instance the following isomorphism, for every object $M$ of $\Cc$: 
$$h_{M}:\Hom_{\Cc}(M\ot M,I)\rightarrow \Hom_{\Cc}(M,M^{*})$$
We give the following definition:
\begin{definition}\delabel{def2}
Let $\Cc$ be a monoidal category satisfying the above assumptions and let $M$ be an object in $\Cc$.
A bilinear form $b:M\ot M\rightarrow I$ is  called \textit{symmetric} if $b\circ c_{M,M}=b$.
It is called \textit{non-singular} if $h_M(b):M\rightarrow M^*$ is an isomorphism. 
\end{definition}
If $b$ is a symmetric bilinear form on an object $M$ in $\Cc$, the couple $(M,b)$ will be called a \textit{quadratic object}
(to keep consistency with the classical notation).
We define morphisms between quadratic objects in the following way: Let $(M,b)$ and $(M',b')$ be two quadratic objects and $f: M\to M'$ a morphism in $\Cc$.
The morphism $f$ is said to be a morphism of quadratic objects if $b' (f\ot f)=b$. 
If $f$ is an isomorphism, it will also be called \textit{isometry}.
\\We have the following lemma:
\begin{lemma}\label{gam}
For $M\in\Cc$ we have a natural isomorphism $M^*\cong \, ^*M$. 
Since the identity functor is naturally isomorphic with $^*(-)^*$, we get a natural isomorphism $\gamma_{M}: M\to (M^{*})^{*}$.
\end{lemma}
\begin{proof}
Consider the evaluation map $\ev:M^*\ot M\to I$ and the coevaluation map $\coev:I\to M\ot M^*$.
Using the naturality of the braiding $c$, combined with the fact that $(M,M^{*},\ev,\coev)$ is a dual pair in $\Cc$,
one checks that $(M^{*},M,\ev \circ c_{M^*,M}^{-1},c_{M,M^*}\circ \coev)$ is also a dual pair in $\Cc$. 
The statement now follows, since the dual of an object is unique up to isomorphism.
\end{proof}
One of the last things we need in order to define the Witt group of $\Cc$ is the following definition and lemma:
\begin{definition}
 Let $M$ be an object of $\Cc$. We say that $M$ is \textit{weakly symmetric} if ${\rm ev}_M={\rm ev}_M\circ c_{M,M^*}\circ c_{M^*,M}$.
\end{definition}
In particular, if the braiding $c$ is symmetric, then every object is weakly symmetric.
We claim the following lemma:
\begin{lemma}\label{wd}
 Assume that $(M,b)$ is a non-singular quadratic object in $\Cc$. Then $M$ is weakly symmetric.
\end{lemma}
\begin{proof}
Since $b$ is symmetric we have that $b\circ c_{M,M}\circ c_{M,M}=b\circ c_{M,M}=b$.
By using the isomorphisms (\ref{dual1}) and (\ref{dual2}), we find that $b={\rm ev}_M\circ (h_M(b)\ot 1)$.
Using the naturality of the braiding and the fact that $h_M(b)$ is an isomorphism, we get 
${\rm ev}_M\circ c_{M,M^*}\circ c_{M^*,M}={\rm ev}_M$, as desired.
\end{proof}

Let now $(M,b)$ and $(M',b')$ be non-singular quadratic objects. We define:

$$(M,b)\perp(M',b'):=(M\oplus M',b\perp b'),$$
where $b\perp b'=b\oplus b'$. One verifies that this is again a non-singular quadratic object.
Let us now denote by $\Qq$ the category of non-singular quadratic objects in $\Cc$ with isometries as morphisms.
Let us denote by $\Dd$ the subcategory of $\Cc$ of all weakly symmetric objects, with only the isomorphisms as morphisms.
Notice that the tensor product of two weakly symmetric objects is not necessarily weakly symmetric again.
The category $\Dd$ is thus not a monoidal category.
It is true, however, that if $M$ is weakly symmetric, then $M^*$ is also weakly symmetric.

We now proceed with the definition of the \textit{hyperbolic functor}.
\\Let $M$ be an object in $\Dd$. We define $b_{M}:(M\oplus M^{*})\ot (M\oplus M^{*})\rightarrow I$ as follows:
$$b_M(M\ot M) = b_M(M^*\ot M^*) = 0$$
$$b_M|_{M^*\ot M} = {\rm ev}_M\textrm{ and } b_M|_{M\ot M^*} = {\rm ev}_M\circ c_{M,M^*}$$
\begin{lemma}
Let $M$ be an object in $\Dd$ and define
$$\Hh(M)=(M\oplus M^{*}, b_{M}).$$
Then $\Hh$ defines a functor from ${\Dd}$ to $\Qq$, called the \textit{hyperbolic functor}. Moreover, $\Hh(M_1\oplus M_2)\cong \Hh(M_1)\perp \Hh(M_2)$.
\end{lemma}
\begin{proof}
A direct verification shows that $h_M(b_M)$ is an isomorphism. The fact that $b$ is symmetric follows from the assumption that $M$ is weakly symmetric.
If $f:M_1\rightarrow M_2$ is an isomorphism in $\Dd$, then $\Hh(f):\Hh(M_1)\rightarrow\Hh(M_2)$ is defined as the direct sum $f\oplus (f^*)^{-1}$.
It is easily checked that $\Hh(f)$ is an isometry.
\end{proof}

Now, in order to define the Witt group, there is one thing left to be checked; namely that $\Hh$ is a cofinal functor. 
We refer to \cite{Quebbemann} for more details about the proof. We mention the fact that $\frac{1}{2}\in k$ is crucial here.

\begin{lemma}\lelabel{invers}
Let $(M,b)$ be a non-singular symmetric quadratic object in $\Cc$. We have the following isometry in $\Qq$:
$$(M,b)\perp (M,-b)\cong \Hh(M).$$

\end{lemma}
\begin{proof}
Since 2 is invertible in $k$, the map $T:M\oplus M\rightarrow M\oplus M$ given symbolically by $$\frac{1}{2}\begin{bmatrix}1 & 1 \\ 1 & -1 \end{bmatrix}$$ is an isomorphism.
This identifies the quadratic object $(M,b)\perp (M,-b)$ with a quadratic object of the form $(M\oplus M,b')$ where $b'$ is given by the matrix
$$\begin{bmatrix} 0 & b \\ b & 0 \end{bmatrix}$$ By using the isomorphism $h_M(b)\ot 1$ and Lemma \ref{wd}, we get the result.
\end{proof}

We thus have an induced map on $K_0$ groups:
$K_0(\Hh):K_0(\Dd)\rightarrow K_0(\Qq)$. We give the following definition:
\begin{definition}
Let $\Cc$ be a $k$-linear rigid braided monoidal category. We define the \textit{Witt group of $\Cc$}, which we denote by $W(\Cc)$, as follows:\index{Witt group}
$$W(\Cc)={K_{0}(\Qq)}/{\rm Im}(K_{0}(\Hh)).$$
\end{definition}
In other words, elements in $W(\Cc)$ are isomorphism classes of non-singular quadratic objects, where we consider the quadratic objects $\Hh(M)$ to be trivial.
The group operation is just $\perp$. By \leref{invers}, the inverse of the class of $(M,b)$ is the class of $(M,-b)$.

Notice that if $F:\Cc\rightarrow \Ee$ is a braided monoidal functor,
and $M$ is a weakly symmetric object in $\Cc$, then $F(M)$ is a weakly symmetric object in $\Ee$.
This follows from the fact that we have a natural isomorphism $F(M^*)\cong F(M)^*$, with respect to which $F({\rm ev}_M)={\rm ev}_{F(M)}$,
and that, if we denote by $\phi_{M,N}:F(M\ot N)\stackrel{\cong}{\rightarrow} F(M)\ot F(N)$ the natural isomorphism associated with $F$, then
$\phi_{M^*,M}\circ F(c_{M,M^*})\circ \phi^{-1}_{M,M^*}= c_{F(M),F(M)^*}$.
As a result of this, we see that the induced functor $\Qq(F):\Qq_{\Cc}\rightarrow \Qq_{\Ee}$ takes hyperbolic objects to hyperbolic objects.
We thus get an induced map of abelian groups $W(F):W(\Cc)\rightarrow W(\Ee)$.
In particular, by considering the case where $F$ is an equivalence of braided categories, we see that the Witt group is an invariant of braided monoidal categories.
\\We have the following lemma, which we shall use later.
\begin{lemma}\prlabel{orde2}
If $k$ contains a primitive fourth root of unity, then every nontrivial element in $W(\Cc)$ has order two.
\end{lemma}
\begin{proof}
Let us take such a primitive fourth root of unity $i$.
Let $(M,b)$ be an object of $\Qq$, and consider $f=i(\textrm{id}_{M})\in \Hom(M,M)$. 
The morphism $f$ is clearly an isomorphism in $\Cc$ and it is easily verified that $-b(f\ot f)=b$,
so $f$ induces an isometry between $(M,b)$ and $(M,-b)$. \leref{invers} now asserts that $(M,b)\perp (M,b)$ is trivial in $W(\Cc)$.
\end{proof}
Now, suppose the braiding $c$ is symmetric. Let $(M,b)$ and $(M',b')$ be two non-singular quadratic objects.
We define
$$(M,b)\ot(M',b')=(M\ot M',b\square b'),$$
where $b\square b'$ is defined to be the composition of the following isomorphisms:
$$b\square b':M\ot M'\ot M\ot M'\stackrel{1\ot c_{M',M}\ot 1}{\longrightarrow} M\ot M\ot M'\ot M'\stackrel{b\ot b'}{\longrightarrow} I$$

Using the fact that $c$ is a symmetry, it is inmediately checked that $(M,b)\ot(M',b')$ is a non-singular quadratic object as well.
One checks that $\ot$ is well-defined on the objects of $K_{0}(\Qq)$ and that it induces a multiplication on the quotient ${K_{0}(\Qq)}/{\rm Im}(K_{0}(\Hh))$
(we refer the reader to \cite[Chapter 5]{Bass}), so $\ot$ induces an associative and a commutative multiplication on $W(\Cc)$,
for which the isometry class of the object $(I,b_{I})$ is the neutral element (here we put $b_{I}:I\ot I\rightarrow I$ the morphism arising from the unit property of $I$).
So this multiplication endows $W(\Cc)$ with the structure of a commutative, unital ring.
We record this fact in the following lemma:
\begin{proposition}
 If the braiding on $\Cc$ is symmetric, then $W(\Cc)$ has a unital commutative ring structure.
\end{proposition}

\section{The semisimple case}\selabel{semisimple}
In this section, we will consider the case where the category $\Cc$ is moreover semisimple, with finitely many isomorphism classes of simple objects.
We shall also assume that the tensor unit is a simple object, and that our ground field $k$ is algebraically closed of characteristic zero.
In other words, we assume that $\Cc$ is a braided fusion category (see \cite{ENO} for precise definitions).
Let $X_1,\ldots, X_n$ be representatives of the irreducible objects of $\Cc$, where $X_1=I$ is the tensor unit.
For any $i=1,\ldots, n$ we will denote by $i^*$ the unique index for which $(X_i)^*\cong X_{i^*}$.
Notice that this map is an involution, that is: $i^{**}=i$.
This follows from the fact that $\Cc$ is semisimple, and is true also when $\Cc$ is not braided.
For every $i=1,\ldots ,n$ we have $$\Hom_{\Cc}(X_{i^*}\ot X_i,I)\cong \Hom_{\Cc}(X_i,X_i).$$
Thus, the morphism ${\rm ev}_i={\rm ev}_{X_i}:X_{i^*}\ot X\rightarrow I$ forms a basis for the vector space $\Hom_{\Cc}(X_{i^*}\ot X_i,I)$.
It follows that ${\rm ev}_i\circ c_{X_i,X_{i^*}}=d_i{\rm ev}_{i^*}$ for some nonzero scalars $d_i$.
Notice that $d_id_{i^*}$ is not necessarily 1. It will be 1 precisely when $X_i$ is a weakly symmetric object.

Let $(M,b)\in\Cc$ be a non-singular quadratic object in $\Cc$.
For $M_i=\Hom_{\Cc}(X_i,M)$, we have the following isomorphism: $$M\cong \oplus_i M_i\ot X_i$$
(The unit object $I$ allows us to consider ${\Vect}_{k}$ as a subcategory of $\Cc$, so we can talk about tensor products with vector spaces freely).
Using the decomposition into isotropic components, the map $b$ can be written as the sum $b=\sum b_i$ where $b_i:(M_{i^*}\ot X_{i^*})\ot (M_{i}\ot X_{i})\rightarrow I$,
and such a map can be uniquely written as the tensor product of ${\rm ev}_i$ and a linear map $f_i:M_{i^*}\ot M_{i}\rightarrow I$.
(there is no ambiguity here since vector spaces commute trivially with every object in the category).
It is easy to see that $h_M(b):M\rightarrow M^*$ will be an isomorphism if and only if $f_i$ is a perfect pairing for every $i$.
The fact that $b$ is symmetric is equivalent to the following equation:
$$ \forall i ~~f_{i^*}= d_i(f_{i}\circ \tau_{i})$$
where $\tau_i:M_i\ot M_{i^*}\rightarrow M_{i^*}\ot M_i$ is the usual flip for vector spaces.
Notice that this already implies that if $M_i\neq 0$, then $X_i$ is weakly symmetric.

We can write $M$ as $\oplus_{i\neq i^*} (M_i\ot X_i)\oplus_{i=i^*} (M_i\ot X_i)$.
Remark that the first direct summand belongs to the image of the hyperbolic functor.
Indeed, if we let $\{i_1,\ldots,i_k\}$ be a set of indices such that $\{i_1,i_1^*,\ldots,i_k,i_k^*\}$
are all the indices which appear in the first direct summand, and $i_r\neq i_s^*$ for any $r,s=1,\ldots, k$,
then it holds that $\Hh(\oplus_{j}M_{i_j}\ot X_{i_j})=\oplus_{i\neq i^*} M_i\ot X_i$.
We thus see that $M$ is equivalent to $\oplus_{i=i^*} M_i\ot X_i$ in the Witt group of $\Cc$.

Let now $i$ be an index which appears in this sum. Since $i=i^*$ and since we have $d_i d_{i^*}=1$ (because $X_i$ is weakly symmetric), 
we deduce that $d_i=\pm 1$ for every $i$ such that $M_i\neq 0$.
We have two options.
If $d_i=-1$, then the map $f_i:M_i\ot M_i\rightarrow k$ must be antisymmetric.
By finding a Lagrangian subspace for $f_i$, we can show that $f_i$ is similar to a matrix of the form
$\begin{bmatrix} 0 & 1 \\ -1 & 0\end{bmatrix}$. But this implies that $M_i$ can be written as $M_i=N_i\oplus T_i$,
and that $M_i\ot X_i$ will be isomorphic with $\Hh(N_i\ot X_i)$.

We are thus left with the case $d_i=1$. In that case $f_i:M_i\ot M_i\rightarrow k$ is a symmetric bilinear map.
Since the Witt group of any algebraically closed field of characteristic zero is $\ZZ_2$ (see \cite{Lam} e.g.), we conclude that the image of $M_i\ot X_i$ in $W(\Cc)$
will only depend on the parity of $\textrm{dim}(M_i)$. This implies that if $(M,q)$ is a quadratic object then $M$ defines $q$ uniquely up to isomorphism.
We thus arrive at the following proposition:
\begin{proposition}
 Let $\Cc$ be as above. The Witt group of $\Cc$ has a basis over $\ZZ_{2}$ given by $\{X_i\}$ where we take all the indices $i$ such that $i=i^*$ and $d_i=1$.
\end{proposition}
Assume now that the braiding on $\Cc$ is symmetric.
Let $K_0(\Cc)$ be the Grothendieck ring of $\Cc$ (over $\ZZ$).
The Witt ring $W(\Cc)$ of $\Cc$ can be considered as the sub-abelian-group of $K_0(\Cc)\ot_{\ZZ}\ZZ_{2}$
spanned by the elements $X_i$ such that $i=i^*$ and $d_i=1$.
Let $p:K_0(\Cc)\ot_{\ZZ}\ZZ_{2}\rightarrow W(\Cc)$ be the projection whose kernel is spanned by all the $X_i$ such that $i\neq i^*$ or such that $i=i^*$ and $d_i\neq 1$.
The restriction of $p$ to $W(\Cc)$ is thus the identity.
We have the following lemma:
\begin{lemma}\label{multiplication}
The multiplication in $W(\Cc)$ is given by the formula $x\cdot y = p(xy)$ where $x,y\in W(\Cc)$.
\end{lemma}
\begin{proof}
This is straightforward, based on the fact that the product in $W(\Cc)$ is given by tensor multiplication, and by the fact that
if $(M,b)$ is a non-singular quadratic object, then $M$ defines $b$ up to an isomorphism.
\end{proof}
Notice that the map $p$ is not necessarily a homomorphism of rings. The formula $x\cdot y = p(xy)$ is true only as long as  $x,y\in W(\Cc)$.

\section{Examples}\selabel{examples}
In this section we will give some examples, arising from finite groups. The field $k$ will be an algebraically closed field of characteristic zero.
Recall that for a finite group $G$ and a $3$-cocycle $\omega\in H^3(G,k^*)$,
the fusion category ${\Vect}_G^{\omega}$ is the category of all vector spaces graded by $G$,
where the associativity constraint is given by $\omega$.
In other words, if $u\in U$ is of degree $g$, $v\in V$ is of degree $h$ and $w\in W$ is of degree $l$,
then the associativity contraint $(U\otimes V)\otimes W\rightarrow U\otimes (V\otimes W)$
will send $(u\ot v)\ot w $ to $\omega(g,h,l)u\ot (v\ot w)$.

The category ${\Vect}_{\ZZ_2}$ has two possible braidings, which we shall denote by $b^0$ and $b^1$.
The braiding $b^0$ will just be the regular symmetric flip of vector spaces.
The braiding $b^1$ will be the map $${b^1}_{V,W}:V\ot W\rightarrow W\ot V$$ $$v\ot w\mapsto (-1)^{|v||w|}w\ot v.$$
In other words, for homogenous elements it is the regular flip, expect in case that the degree of both $v$ and $w$ is 1, and then it is minus the flip
(we write here the group $\ZZ_2$ additively).
The Witt group arising from $b_0$ will have two basis elements, one arising from each element of the group $\ZZ_2$.
The Witt group of $b_1$ will only have one basis element, arising from the trivial element of the group.
This is because the simple non-unit object $X_2$ has $d_2=1$ in the first case and $d_2=-1$ in the second case.

Let now $\omega\in H^3(\ZZ_2,k^*)$ be the only nontrivial class.
We can write $\omega$ in the following form: $\omega(g,h,l)=1$ unless $g=h=l=1$ and then $\omega(1,1,1)=-1$.
The category ${\Vect}_{\ZZ_2}^{\omega}$ also has two braidings.
The first braiding $b^i$ is given by $$b^i:V\ot W\rightarrow W\ot V$$ $$v\ot w\mapsto i^{|v||w|}w\ot v,$$
and the second braiding $b^{-i}$ is defined using $-i$ instead of $i$ ($i$ satisfies $i^2=-1$).
Both braidings are not symmetric, and both Witt groups will have only one basis element.
The simple non-unit object $X_2$ is not weakly symmetric in this case.
\\\\If $G$ is any finite group, then $\Rep{G}$, the category of finite-dimensional complex representations of a finite group $G$, is a symmetric fusion category, where the symmetry is just the regular flip, denoted $\tau$.
Let $V$ be an irreducible representation. Assume that $V\cong V^*$. Then we have a unique $G$-homomorphism (up to scaling) $d_V:V\ot V\rightarrow k$.
The map $d_V\circ\tau_{V,V}$ is also a $G$-homomorphism. Since $\tau_{V,V}^2=1$, we have that $d_V\circ \tau_{V,V}=\nu_2(V) d_V$ where $\nu_2(V)=\pm 1$.
We say that $V$ is \textit{symmetric} in case $\nu_2(V)=1$, and we say that $V$ is \textit{anti-symmetric}
in case $\nu_2(V)=-1$. The Witt group will then be the $\ZZ_{2}$ vector space with all the irreducible representations $V$ with $V\cong V^*$ and $\nu_2(V)=1$ as a basis.
The multiplication in $W(\Cc)$ will be induced by the multiplication in the character ring of $G$, by using Lemma \ref{multiplication}.
We remark that $\nu_2$ is called the second Frobenius-Schur indicator.
If the character of $V$ is $\psi$, the Frobenius-Schur indicator can be calculated as:
$$\nu_2(V)=\frac{1}{|G|}\sum_{g\in G}\psi(g^2).$$
It is known that if $V\ncong V^*$ then $\nu_2(V)=0$. If $V\cong V^*$, then $\nu_2(V)=\pm 1$ according to the criterion written above.
For more details, see \cite{GorLie}.
\\\\More generally, assume that $G$ is a finite group and that $u\in G$ is a central element of order 2.
Then we can define a new symmetric braided structure on ${\Rep}G$.
If $V$ and $W$ are two irreducible representations, then $u$ acts on $V$ by a scalar $u_V$ and on $W$ by a scalar $u_W$.
We have that $u_V=\pm 1$ and $u_W=\pm 1$. We define the new braiding $c_{V,W}:V\ot W\rightarrow W\ot V$
to be the regular flip if $u_V=1$ or $u_W=1$, and $-1$ times the flip in case $u_V=u_W=-1$. Denote this braided fusion category by $\Rep{G}_{u}$
The new braiding can be written as $c_{V,W}(v\ot w) = \frac{1}{2}(1\ot 1 + 1\ot u + u\ot 1 - u\ot u)(w\ot v)$.
Notice that the Witt group of $\Rep{G}_u$ with that braiding might be different from the Witt group with the usual braiding.
The basis of the Witt group will consist now of the irreducible self-dual representations with $\nu_2(V)u_V=1$.
\\\\If $G$ is any finite group, we can also consider the category of representations of the Drinfeld double $D(G)$ of $G$.
A representation of $D(G)$ will be a $G$-graded vector space $V=\oplus_{g\in G}V_g$ which is also a $G$-representation,
such that $g\cdot V_h\subseteq  V_{ghg^{-1}}$.
The category $\Rep{D(G)}$ is braided. If $V$ and $W$ are two objects of $\Rep{D(G)}$,
and $v\in V$ is of degree $g$, then the braiding is given by $c_{V,W}(v\ot w) = gw\ot v$.
Any irreducible representation $V$ of $D(G)$ will be supported on a single conjugacy class $[g]$,
and will be induced from an irreducible representation $\tilde{V}$ of $C_G(g)$:
$V=kG\ot_{kC_G(g)} \tilde{V}$. The degree of $x\ot v$ will be $xgx^{-1}$.
Such an irreducible representation of $D(G)$ will be self-dual if and only if the following condition holds:
there exists an element $h\in G$ such that $hgh^{-1}=g^{-1}$,
and the irreducible representation $h(\tilde{V})$ of $C_G(g)$ (which is the vector space $\tilde{V}$ upon which $x\in C_G(g)$
acts as $hxh^{-1}$) is isomorphic with $\tilde{V}^*$. The last condition can be understood in the following way:
there is a (unique up to scaling) $C_G(g)$-homomorphism $f:\tilde{V}\otimes \tilde{V}\rightarrow k$ which satisfies:
\begin{equation}\label{strangesymmetry}f(v\ot w) = f(xv\ot h^{-1}xh w)\end{equation} for every $v,w\in \tilde{V}$ and $x\in C_G(g)$.
The $D(G)$-map $V\ot V\rightarrow k$ will then be defined in the following way:
$$b((x\ot v)\ot(y\ot w)) = 0\textrm{ if } h^{-1}x^{-1}y\notin C_G(g)\textrm{, and} $$ $$b((x\ot v)\ot (y\ot w)) = f(v\ot h^{-1}x^{-1}yw)\textrm{ if } h^{-1}x^{-1}y\in C_G(g).$$
The pairing $b$ will then be symmetric if and only if $f$ will satisfy the equation $$f(v\ot w) = f(w\ot h^{-2}g^{-1}v).$$
In general, the function $\hat{f}(v\ot w)= f(w\ot h^{-2}g^{-1}v)$ will still satisfy Equation \ref{strangesymmetry}, 
so it will be equal to $f$ up to a nonzero scalar (since $V$ is irreducible).

As an example of this, we can consider the case where the group $G$ is abelian.
In that case we have that irreducible $D(G)$-representations correspond to pairs $(g,\psi)$ where $g\in G$ and $\psi\in\hat{G}$, the character group of $G$.
Such a representation will be self-dual if and only if $g^2=1$ and $\psi^2=1$. 
It will be symmetric if and only if $\psi(g)=1$.

\begin{remark}
\textcolor{black}{More generally, let $H$ be a finite-dimensional semisimple braided Hopf algebra over an algebraically closed field $k$ of characteristic $0$. 
Then $\Rep{H}$ is a braided fusion category and one can consider in particular self-dual objects in this category, endowed with a quadratic form, in the sense of \deref{def2}. 
Notice that this does not coincide with the notion of self-dual $H$-module endowed with a ``symmetric non-singular bilinear $H$-invariant form'' in the sense of \cite{LinMon}, 
unless the braided structure is the symmetric structure.}
\end{remark}

\section{An invariant for Isocategorical Groups}\selabel{IsoCat}
Let $G$ and $H$ be two finite groups. Recall the following definition:
\begin{definition}[\cite{EtiGel}, Definition 1.1]
$G$ and $H$ are called isocategorical if $\Rep{G}$ is monoidal equivalent to $\Rep{H}$ (without regard for the symmetric structure).
\end{definition}
As already mentioned in the Introduction, isocategorical groups need not be isomorphic.
In spite of this fact, all finite groups isocategorical to a given group can be explicitly classified in group-theoretical terms (cf. \cite[Theorem 1.3]{EtiGel}).
In this section, relying on this description, we will prove that the Witt ring is an invariant for finite isocategorical groups, i.e.
\begin{theorem}\thlabel{IsoCatWitt}
Let $G$ and $H$ be two finite, isocategorical groups. Then $W({\Rep}G)$ and $W({\Rep}H)$ are isomorphic commutative, unital rings.
\end{theorem}
In order to prove \thref{IsoCatWitt}, we first briefly recall the result \cite[Theorem 1.3]{EtiGel}, which describes all groups isocategorical to a given group $G$.
Let $A$ be a normal abelian subgroup of $G$ of order $2^{2m}$ for some $m\in \mathbb{N}$.
We have an induced action of $Q:=G/A$ on $A$.
Let $[\alpha]\in H^2(A^{\vee},k^{\times})^Q$ be a $Q$-invariant cohomology class (we denote by $A^{\vee}$ the character group of $A$).
Then for every $q\in Q$, $q(\alpha)/\alpha$ is a trivial two-cocycle. 
We thus choose, for every $q\in Q$, a one-cochain $z(q):A^{\vee}\rightarrow k{^\times}$ such that \begin{equation}\label{eq2}\partial(z(q))= q(\alpha)/\alpha.\end{equation}
A direct calculation shows that for every $p,q\in Q$ the cochain \begin{equation}\label{eq1}b(p,q):=\frac{z(pq)}{z(p)p(z(q))}\end{equation} has trivial coboundary.
It is thus a character of $A^{\vee}$ and therefore belongs to $A^{\vee \vee}\cong A$. 
We can thus think of $b$ as a two-cocycle of $Q$ with values in the $Q$-module $A$.
We define now a new group which we denote by $G_b$.
As a set, $G_b=G$, and the multiplication is given by $$g\cdot_b h  = b(\bar{g},\bar{h})gh.$$
In other words, if we write $G$ as an extension
$$E:1\rightarrow A\rightarrow G\rightarrow Q\rightarrow 1,$$
then $E$ gives rise to a class $[E]\in H^2(Q,A)$. The group $G_b$ can be written similarly, but the resulting two-cocycle will be $E\cdot b$.
\begin{theorem}[\cite{EtiGel}, Theorem 1.3]
 The group $G_b$ is isocategorical with $G$, and any group isocategorical with $G$ is obtained by this construction.
\end{theorem}
\begin{remark} \relabel{IzuKosaxample}
In \cite{IzuKosa}, an (infinite) family of pairs of nonisomorphic, yet isocategorical groups $G^{m}$ and $G^{m}_{b}$ ($m\in \NN$, $m>2$) of order $2^{2m}$ is presented.
It turns out that $G^{3}$ and $G^{3}_{b}$ are the smallest known example of two non-isomorphic, yet isocategorical groups that appear in literature. 
In this example of order 64, the group $A$ is the direct product $\ZZ_4 \times \ZZ_4$ (we denote its generators by $a_1$ and $a_2$). 
The group $Q$, which is isomorphic to the Klein group $\ZZ_2 \times \ZZ_2$ and whose generators we denote by $q_1$ and $q_2$, 
acts on $A$ as follows (addition in indices is modulo 2):
$${q_i}(a_i)=a_i~~;~~{q_i}(a_{i+1})={a_{i}^{2}{a_{i+1}}}.$$ 
The group $G^{3}$ is the semidirect product $G^{3}=Q\ltimes A$.
The cocycle $b$ is given by the formula
$$b(q_1^{t_1}q_2^{t_2},q_1^{r_1}q_2^{r_2}) = a_1^{l_1}a_2^{l_2}$$
where $l_i = \delta_{1,t_i}\delta_{1,r_i}$.
Actually, in all these examples, the class $[E]\in H^2(G^m/A,A)$ is trivial (as these groups $G^{m}$ are semi-direct products),
but for general isocategorical groups this need not to be the case.
\end{remark}
\subsection*{Proof of \thref{IsoCatWitt}}\hfill\\
In order to prove that $W(\Rep G)\cong W(\Rep H)$, we will construct explicitly the monoidal equivalence $F:\Rep G\rightarrow \Rep H$.
We will then show that for every irreducible $V\in \Rep G$ we have that $\nu_2(V)=\nu_2(F(V))$.
Since $F$ induces an isomorphism of the Grothendieck rings of the two categories, it will follow from the results of \seref{semisimple} that the assignment $V\mapsto F(V)$
induces an isomorphism of the Witt rings. 

We know that $H\cong G_b$ for some two-cocycle $b$ arising from the construction above. 
Let $A$, $Q$, $\alpha$, $z$ and $b$ be as in the construction.
We define $F:\Rep G\rightarrow \Rep G_b$ in the following way:
$F(V) = V$ as a vector space. To write the action of $G_b$, let $V=\oplus_{\phi\in A^{\vee}} V_{\phi}$ 
be a direct sum decomposition of $V$ with respect to the characters of $A$.
The action of $g\in G_b$ on $v\in V_{\phi}$ will be given by 
$$g\cdot_b v = z^{-1}(\bar{g})(\bar{g}(\phi))gv.$$
A direct calculation (based on Equation \ref{eq1}) shows that this really defines an action of $G_b$ on $G$.
It is clear that $F$ is an equivalence of categories.
We would like to describe the monoidal structure of $F$ as well.
In other words, we need to find a canonical isomorphisms $f_{V,W}$ between the representations $F(V\ot W)$ and $F(V)\ot F(W)$ of $G_b$.
Let $v\in V_{\phi}$ and $w\in W_{\psi}$ for some $\phi,\psi\in A^{\vee}$.
Then this isomorphism will be given by $$f_{V,W}(v\ot w) = \alpha^{-1}(\phi,\psi)v\ot w$$
Again, a direct calculation (based on Equation \ref{eq2}) shows that this furnishes a monoidal structure on $F$.
Now assume that $V$ is a self-dual irreducible object in $\Rep G$. 
Then we have a $G$-map $r:V\ot V\rightarrow k$.
Since this is in particular an $A$-map, we see that $r(V_{\phi}\ot V_{\psi})=0$ unless $\phi\psi=1$.
We would like to show that $\nu_2(V)=\nu_2(F(V))$.
A direct calculation, using the isomorphism $f$ above, shows that $$\nu_2(F(V))= \alpha(\phi,\phi^{-1})\alpha^{-1}(\phi^{-1},\phi)\nu_2(V)$$
where $\phi$ is any character of $A$ for which $V_{\phi}\neq 0$.
Since $\alpha$ is a two-cocycle, we have that $\alpha(\phi,\phi^{-1})=\alpha(\phi^{-1},\phi)$.
This follows easily from the fact that $U_{\phi}$ and $U_{\phi^{-1}}$ commute in the twisted group algebra $k^{\alpha}A^{\vee}$.
We have that $\nu_2(V)=\nu_2(F(V))$, and we are done.
\begin{remarks}
{\textcolor{white}{Ha ha ha. Writing in white.}}\\
\begin{itemize}
\item If $G$ and $H$ are isocategorical, then $\Rep{D(G)}$ and $\Rep{D(H)}$ are braided equivalent monoidal categories, 
and thus we also have an isomorphism of abelian groups $W(\Rep{D(G)})\cong W(\Rep{D(H)})$.
\item As aimed at in the Introduction, \thref{IsoCatWitt} can also be deduced from Theorem 4.1 in \cite{MasonNg}, 
which states that finite-dimensional semisimple quasi-Hopf algebras $K$ and $K'$ (over $\CC$) having tensor equivalent categories of finite-dimensional representations, 
necessarily have identical families of Frobenius-Schur indicators. This is a more general result than the one presented above. However, the proof we give here, 
using the group-theoretical description of finite isocategorical groups, could be considered to be more explicit (than the cited Theorem of Mason and Ng's) in case $K=\CC G$ and $K'=\CC H$.
\end{itemize}
\end{remarks}

\section{Application}\selabel{Application}
In this section, we apply \thref{IsoCatWitt} to obtain examples of finite groups that have isomorphic character rings, yet are not isocategorical.
The result just below \cite[Corollary 1.4]{EtiGel} already gives us a powerful result to detect such groups.
Let us recall a finite group $G$ is called \textit{categorically rigid} if any group isocategorical to $G$ is actually isomorphic to $G$.
\begin{corollary}[Etingof-Gelaki]\colabel{Corollary1.4}
If $G$ is not categorically rigid then $G$ admits a normal abelian subgroup $A$, of order $2^{2m}$, $m\geq 1$, equipped with a skew-symmetric $G$-invariant isomorphism $R:A^{\lor}\to A$.
In particular, all groups of order $2k+1$, $2(2k+1)$, and all simple groups are categorically rigid.
\end{corollary}
Now, the description of $W(\Rep{G})$ from \seref{examples} allows us to compute $W(\Rep{G})$ in a rather quick fashion 
(where one can use a computer algebra system such as GAP if one wishes to do so).
As a first example, consider the two non-abelian groups of order 8: 
$$D_{8}=\langle a,b|a^{4}=b^{2}=1,b^{-1}ab=a^{-1}\rangle~ \textrm{and} ~Q_{8}=\langle a,b|a^{4}=1,b^{2}=a^{2},b^{-1}ab=a^{-1}\rangle.$$
These groups have isomorphic character rings. Both groups have 4 irreducible representations of degree 1,
their corresponding characters being denoted by $\chi_{i}^{D}$ and $\chi_{i}^{Q}$ respectively ($i=1,2,3,4$), and 1 irreducible representation of degree 2, denoted resp.
$\chi_{5}^{D}$ and $\chi_{5}^{Q}$. One easily computes that $\nu_{2}(\chi_{i}^{D})=1$ for $i=1,2,3,4,5$ and that $\nu_{2}(\chi_{i}^{Q})=1$ for $i=1,2,3,4$, but $\nu_{2}(\chi_{5}^{Q})=-1$, which implies that $W({\Rep}D_{8})$ and $W({\Rep}Q_{8})$ are not isomorphic. So we recover the fact that $D_8$ and $Q_8$ are not isocategorical
(which is a special case of a result first proven in \cite{TambaraYam}).
Remark that this result also follows from \coref{Corollary1.4}, as all normal subgroups $A$ of order $4$ of $Q_8$ are cyclic, and hence do not admit any skew-symmetric form $R:A^{\lor}\to A$.
\\The groups $D_8$ and $Q_8$ are by no means rare examples of non-isomorphic groups having the same character table. In fact, there are many more examples of this sort among $2$-groups
(actually, Davydov \cite{Davydov} proved that finite groups have isomorphic complex character rings if and only if their group algebras are twisted forms of one another).
\subsection{Order 16}\selabel{Order 16}
It is known that, up to isomorphism, there are precisely $9$ non-isomorphic, non-abelian groups of order $16$ (cf. \cite{GorLie}, pages 304-307 e.g.).
Amongst those -and besides $D_8 \times \ZZ_2$ and $ Q_8 \times \ZZ_2$- the following two pairs of groups have isomorphic character rings, the first couple being
$$G_1=\langle a,b|a^{8}=b^{2}=1,b^{-1}ab=a^{-1}\rangle\cong D_{16}~~~\textrm{and}~~~G_2=\langle a,b|a^{8}=1,b^{2}=a^{4},b^{-1}ab=a^{-1}\rangle,$$
and the other such pair of groups:
$$G_{3}=\langle a,b,z|a^{4}=1,b^{2}=z,z^{2}=1,b^{-1}ab=az\rangle~~~\textrm{and}~~$$
$$G_4 =\langle a,b,z|a^{4}=1,b^{2}=z^{2}=1,b^{-1}ab=az, za=az,bz=zb\rangle.$$
We first take a look at $G_1$ and $G_2$. Let $\chi$ be the irreducible character of degree $2$, for which $\chi(a^{4})=2$, let $\psi$ be the irreducible character of degree $2$ for which
$\psi(a)=\sqrt{2}$ and let $\kappa$ be the irreducible character of degree $2$, for which $\kappa(a)=-\sqrt{2}$.
Then one computes that
$$\nu_{2}(\chi)=\nu_{2}(\psi)=\nu_{2}(\kappa)=+1 ~~\textrm{for}~G_{1}~~\textrm{and}~\nu_{2}(\chi)=+1 ~~\textrm{for}~G_{2}~\textrm{but}~\nu_{2}(\psi)=\nu_{2}(\kappa)=-1.$$
We thus have that $W(\Rep{G_1})\not\cong W(\Rep{G_2})$, so $G_1$ and $G_2$ are not isocategorical.
Remark that this result can also be obtained from \coref{Corollary1.4}. Indeed, $G_1$ and $G_2$ are both categorically rigid since all of their normal subgroups
are cyclic (actually, they are the only non-abelian groups of order $16$ having this property).
\\Now, let us turn our attention to the groups $G_3$ and $G_4$.
Let $\chi$, resp. $\psi$, be the irreducible characters of degree $2$ such that $\chi(a^2)=2$, resp. $\psi(a^{2})=-2$. Then one computes that
$$\nu_{2}(\chi)=+1~~;~~\nu_{2}(\psi)=-1 ~~\textrm{for}~G_{3}~~\textrm{and}~\nu_{2}(\chi)=\nu_{2}(\psi)=+1 ~~\textrm{for}~G_{4}.$$
So here as well, \thref{IsoCatWitt} allow us to conclude that $G_3$ and $G_4$ are not isocategorical.
Observe that this result can not immediately be deduced from \coref{Corollary1.4}, as both groups have the Klein group $\ZZ_2\times \ZZ_2$ as a normal subgroup and hence, since
$H^{2}(\ZZ_2\times \ZZ_2,\CC^{*})\cong \ZZ_2$, they do admit a skew-symmetric isomorphism from $\ZZ_2\times \ZZ_2$ to its character group $(\ZZ_2\times \ZZ_2)^{\lor}$.
Finally, consider $D_8\times \ZZ_2$ and $Q_8\times \ZZ_2$; they are non-isocategorical as well, as can be checked as well by computing their Witt rings.
\\Actually, the ring of characters combined with the Witt ring, together with \coref{Corollary1.4}, are sufficient to decide that all groups of order less than $32$ 
are categorically rigid, which is a known result 
(cf. \cite[Theorem 6.3]{Shimizu} e.g. for an alternative proof of this fact). It is then natural to ask the following question: Let $G$ and $H$ be two finite groups. Suppose that $K_0(\Rep{G})\cong K_0(\Rep{H})$, $W(\Rep{G})\cong W(\Rep{H})$ 
and that both groups have the same number of self-dual irreducible representations. Does this imply that $G$ and $H$ are isocategorical?
The answer to this question is negative. Shimizu has proved in \cite{Shimizu} that if $G$ and $H$ are isocategorical, 
then the number of elements of order $n$ in $G$ is equal to the number of elements of order $n$ in $H$, for every $n$.
We will give here an example of a pair of groups which have isomorphic Grothendieck rings, isomorphic Witt rings and the same number of irreducible self-dual representations, 
but do not have the same number of elements of order 4, and are therefore not isocategorical.
The following more specific question then arises: assume that $G_1$ and $G_2$ have isomorphic Grothendieck rings, isomorphic Witt rings,
the same number of self-dual irreducible representations, and the number of elements of order $n$ in $G_1$ and $G_2$ is the same for every $n$. 
Does that mean that $G_1$ and $G_2$ are isocategorical? We will give an example here below which proves that the answer to that question is also negative.

\subsection{Order 32}\selabel{Order 32}
Among the non-abelian groups of order $32$ we have the following two groups
$$G:=(({\ZZ}_{4} \times \ZZ_{2}) \rtimes \ZZ_{2}) \rtimes \ZZ_{2}~~~~\textrm{and}~~~~H:=(\ZZ_{8} \rtimes \ZZ_{2}) \rtimes \ZZ_{2}$$
(which can be defined using GAP's SmallGroup function as $G=\textrm{SmallGroup}(32,6)$ and $H=\textrm{SmallGroup}(32,7)$).
One checks that $K_0(\Rep{G})\cong K_0(\Rep{H})$, $W(\Rep{G})\cong W(\Rep{H})$ and that both groups have precisely 7 self-dual irreducible representations.
However, $G$ has 20 elements of order 4, whereas $H$ has only 4 elements of this order, so by Theorem 1.1 from \cite{Shimizu}, $G$ and $H$ cannot be isocategorical.
\\In fact, there is only one pair of groups $G_1,G_2$ of order 32 such that $K_0(\Rep{G_1})\cong K_0(\Rep{G_2})$, 
$W(\Rep{G_1})\cong W(\Rep{G_2})$ and such that both groups have the same number of self-dual irreducible representations 
and the same number of elements of a given order. In terms of generators and relations, one can describe these groups as follows 
(these groups correspond to $\textrm{SmallGroup}(32,27)=G_1$ and $\textrm{SmallGroup}(32,34)=G_2$):
$$G_1:=\langle a,b,c,d,e| a^2,b^2,c^2,d^2,[a,b],[a,c],[a,d],[b,c],[b,d],[c,d], [e,c]=a,[e,d]=b\rangle$$
$$G_2:=\langle a,b,c| a^4,b^4, c^2,[a,b], [c,a]=a^2,[c,b]=b^2\rangle.$$
One can easily check that $G_1\cong \ZZ_2\ltimes (\ZZ_2)^4$, where the first $\ZZ_2$ copy is generated by $e$ and $(\ZZ_2)^4$ is generated by $a,b,c,d$,
and that $G_2\cong \ZZ_2\ltimes (\ZZ_4)^2$ where the first copy of $\ZZ_2$ is generated by $c$ and $(\ZZ_4)^2$ is generated by $a$ and $b$.
Also, one can check that both groups have isomorphic Witt rings, Grothendieck rings, and that the number of irreducible self dual representation is 10 in both cases.

We claim that the two groups are not isocategorical. We will check this directly, by considering all normal abelian subgroups of order 4 and 16.
(

Subgroups of order 4:\\
Assume that $G_2$ has a normal subgroup $H$ of order 4.
If $H$ contains an element of the form $ca^ib^j$ for some $i,j$, then by using the normality of $H$ we can see that $H$ necessarily contains $a^2$ and $b^2$.
But then the order of $H$ is at least 8 which is a contradiction.
The subgroup $H$ is thus contained in the subgroup generated by $a,b$.
We only care about abelian subgroups which have a nontrivial second cohomology group.
In the case of groups of order 4, this means that we must have $H=\ZZ_2\times \ZZ_2$.
But there is only one subgroup of $\langle a,b\rangle$ of this form, namely $\langle a^2,b^2\rangle$.
Since this subgroup is central in $G_2$, we will not get any new groups isocategorical to $G_2$ this way.

Subgroups of order 16:\\
A subgroup $H$ of order 16 of $G_2$ will necessarily contain $a^2$ and $b^2$.
The subgroup $H$ must contain an element of the form $a^ib^j$ where $i$ or $j$ are odd.
If $H$ contains also an element of the form $ca^kb^l$, then $H$ is not abelian.
We therefore conclude that $H$ is contained in $\langle a,b\rangle$.
But since this subgroup has order 16 we conclude that $H=\langle a,b\rangle$.
In other words: the only subgroup of $G_2$ of order 16 that might be relevant for us is $\langle a,b\rangle$.

Let now $H$ be a subgroup of order 16 of $G_1$.
As can easily be seen, $H$ must contain $a$ and $b$.
Similar to the $G_2$ case, $H$ cannot contain an element of the form $ea^ib^jc^kd^l$, as otherwise it would not be abelian.
So $H$ is contained in $\langle a,b,c,d\rangle$ and since they have the same order, $H=\langle a,b,c,d\rangle$.
So the only abelian normal subgroup of $G_1$ of order 16 is $\langle a,b,c,d\rangle$.
But since this subgroup of $G_1$ is not isomorphic with the only abelian normal subgroup of $G_2$, they cannot be isocategorical via this subgroup. We are done.
\\\\One verifies, again only using the ring of characters, the Witt ring and \coref{Corollary1.4}, that all groups of order greater than 32 and less than 64 are categorically rigid.
We can thus resume the results of this section in the following proposition:
\begin{proposition}
All groups of order less than 64 are categorically rigid.
\end{proposition}
This makes the groups $G^{3}$ and $G^{3}_{b}$ considered in \reref{IzuKosaxample} indeed the smallest possible example of isocategorical, yet non-isomorphic groups.
\subsection*{Acknowledgements}
The first named author was supported by the grant G.0117.10 from the Fund for Scientific Research-Flanders (Belgium) (F.W.O.-Vlaanderen). 
He would also like to thank Philippe Cara for fruitful discussions.
The second author was supported by the Danish National Research Foundation through the Centre for Symmetry and Deformation.

\end{document}